\documentclass[12pt,a4paper]{amsart}
\usepackage[T1]{fontenc}
\usepackage{inputenc}
\usepackage{lmodern}

\usepackage{amssymb}
\usepackage{enumitem}

\newcommand{\bbR}{\mathbb R}
\newcommand{\bbN}{\mathbb N}

\renewcommand{\leq}{\leqslant}

\newtheorem{thm}{Theorem}
\newtheorem{lem}[thm]{Lemma}

\newtheorem{pro}[thm]{Proposition}

\date{}
\title{Center of distances and Bernstein sets}

\author[M. Kula]{Mateusz Kula}
\address{Institute of Mathematics, University of Silesia in Katowice, Bankowa 14, 40-007 Katowice}
\email{mateusz.kula@us.edu.pl}

\begin{document}
\begin{abstract}
	We show that for any subset $A\subset [0,\infty)$, where $0\in A$, there exists a Bernstein set $X\subset \bbR$ such that $A$ is the center of distances of $X$.
\end{abstract}
\maketitle
\section{Introduction}
The notion of center of distances has been introduced in \cite{bpw} as follows. If $(X,d)$ is a metric space, then the set
$$S(X):=\{a\in[0,\infty)\colon \forall_{x\in X}\exists_{y\in X}d(x,y)=a\}$$
 is called the \textit{center of distances} of $X$.
Although this notion is quite general, most applications concern subsets of real numbers $\bbR$.
 For example, the center of distances has been successfully used in proofs that some sets are not achievement sets of any series, see \cite{bpw} or \cite{bgm}. Exact calculation of the center of distance is sometimes complicated, for example for central Cantor sets, see \cite{bbfpw} and \cite{Banakiewicz}.  
 
 In \cite{bfhlpw}, the question to characterize sets $S(X)$ for subsets of the real line has been raised. At the conferences  `Inspirations in Real Analysis II' (2024) and `46th Summer Symposium in Real Analysis' (2024), M. Filipczak asked the following.

\noindent \textbf{Question.} Given a subset $A\subset [0,\infty)$ with $0\in A$, does there exist $X\subset \bbR$ such that $S(X)=A$?

We give a positive answer to Filipczak's question, see Theorem \ref{thm1}, and improve this result showing that in addition $X$ can be chosen to be a Bernstein set, see Theorem \ref{thm2}. It is worth noting that the answer to Filipczak's question was already known for compact $A$ and in that case $X$ can be chosen to be closed, see \cite[Corollary 4.11]{bfhlpw}.

We use the standard notation. Throughout this paper $2^{\leq n}$, $2^{<n}$ and $2^n$ denote the set of all finite binary sequences of length $\leq n$, $<n$ and $=n$, respectively.
If $s=(s_1,\ldots,s_n)$ and $t=(t_1,\ldots,t_m)$ are finite sequences, then, as usual, $s^\frown t$ is the concatenation $(s_1,\ldots, s_n, t_1,\ldots, t_m)$.  Also, $\operatorname{lin}_{\mathbb Q}(X)$ denotes the linear span of $X$ over the rationals $\mathbb Q$.
\section{On the average operator}

Given a subset $C\subset \bbR$, let $T_C\colon \mathcal P(\mathbb R)\rightarrow\mathcal P(\bbR)$ be the operator given by the formula
$$T_C(Y):=\left\{\frac12(y+y')\colon y,y'\in Y\right\} \setminus C, \text{ where } Y\subset \bbR.$$ 
Let $$T^n_C(Y):=\underset{n \text{ times}}{\underbrace{(T_C\circ\dots\circ T_C)}}(Y)\, \text{ and }\, T^\infty_C(Y):=\bigcup_{n\in\bbN}T^n_C(Y).$$ 
Note that if $C\cap Y=\emptyset$, then
$Y\subset T_C^\infty (Y)$, $C\cap T_C^\infty (Y)=\emptyset$ and
\begin{equation}\label{eq1}T_C(T^\infty_C(Y))=T^\infty_C(Y).\end{equation}
Indeed, if $v\in T_C(T^\infty_C(Y))$, then there exist $n\in\bbN$ and $u,u'\in T^n_C(Y)$ such that $v=\frac12(u+u')\notin C$, and hence $v\in T_C^{n+1}(Y)\subset T_C^\infty (Y)$.

\begin{lem}\label{lem1}
	Given $C\subset \bbR$, $Y\subset \bbR\setminus C$ and $v\in T^n_C(Y)$, there exists a function $d\colon 2^{\leq n}\rightarrow \bbR$ 
	such that \begin{enumerate}[label=\textnormal{(\Roman*)}]
		\itemsep1pt
				\item \label{item1} $d({\emptyset})=v$;
			\item \label{item2} $d(s)=\frac12(d(s^\frown 0)+d(s^\frown 1))$ for $s\in 2^{<n}$;
			\item \label{item3}  $d(s)\in Y$ for $s\in 2^n$;
		\item \label{item4} $d(s)\notin C$ for $s\in2^{\leq n}$.
	\end{enumerate}
Moreover, if $Y=\{x-b,x+b\}$ and $x\in C$, then $d(s)=x-b$ for all $s\in 2^{\leq n}$ or $d(s)=x+b$ for all $s\in 2^{\leq n}$. 
	\end{lem}
\begin{proof}
	Let $C\subset \bbR$ and $Y\subset \bbR\setminus C$. The proof is by induction on $n\in\bbN$. If $n=0$, there is nothing to do. Fix $n\in\bbN$ and assume the hypothesis is fulfilled for $n$.
	Fix $$v\in T^{n+1}_C(Y)=T_C(T^n_C(Y)).$$ By the definition of $T_C$ there exist $v_0,v_1\in T^n_C(Y)$ such that $v=\frac12(v_0+v_1)\notin C$. By induction hypothesis there exist functions $d_0,d_1\colon 2^{\leq n}\rightarrow \bbR$ such that $d_i(\emptyset)=v_i$ and conditions \ref{item2}--\ref{item4} with $d_i$ in place of $d$ are fulfilled.  Define $d\colon 2^{\leq n+1}\rightarrow\bbR$ such that $d(\emptyset)=v$, $d({0^\frown s})=d_0(s)$ and $d({1^\frown s})=d_1(s)$ for any $s\in 2^{\leq n}$.
	
	If $Y=\{x-b,x+b\}$ and $x\in C$, then $T^n_C(Y)=Y$. Thus the function $d$ must be constant.
\end{proof}
In particular, if $v\in T^n_C(Y)$, then $$v=\frac1{2^n}\sum_{s\in 2^n}d(s),$$ where $d\colon 2^{\leq n}\rightarrow \bbR$ is as in Lemma \ref{lem1}.
\section{Surjectivity of the center of distances operator}
Now we prove a theorem, which answers Filipczak's question. 
\begin{thm}
	\label{thm1}For any set $A\subset [0,\infty)$, where $0\in A$, there exists $X\subset \bbR$ such that $S(X)=A$.
	\end{thm}
\begin{proof}
Let $0\in A\subset [0,\infty)$ and consider an enumeration
	$$[0,\infty)\setminus A = \{b_\alpha\colon \alpha<\kappa\},$$
	where $\kappa=\operatorname{card}([0,\infty)\setminus A)\leq\mathfrak c$. Any Hamel basis of $\bbR$ over $\mathbb Q$ has cardinality $\mathfrak c$, so we inductively choose a transfinite sequence $(x_\alpha)_{\alpha<\kappa}$ of real numbers such that
	\begin{equation} \label{eq2} x_\alpha\notin \operatorname{lin}_{\mathbb Q}\left(\{b_\beta\colon \beta\leq \alpha\}\cup \{x_\beta\colon \beta<\alpha\}\right).\end{equation}
	Put $C:=\{x_\alpha\colon \alpha<\kappa\}$, $Y:=\bigcup_{\alpha<\kappa}\{x_\alpha+b_\alpha, x_\alpha-b_\alpha\}$ and
	$$X:=\bbR\setminus T_C^\infty\left(Y\right).$$

	 By definitions, $C\subset X$ and for each $\alpha<\kappa$ the number $x_\alpha\in X$ witnesses that $b_\alpha\notin S(X)$, therefore $$S(X)\subset [0,\infty)\setminus \{b_\alpha\colon \alpha<\kappa\}= A.$$

	To prove the opposite inclusion, fix $b\in [0,\infty)\setminus S(X)$. Then there exists $x\in X$ such that $x+b,x-b\notin X$.
	Since $$x\notin T_C^\infty(Y)=T_C(T_C^\infty(Y))$$ (recall formula \eqref{eq1}) and $x$ is the middle of points $x+b,x-b\in T_C^\infty(Y)$, it follows that $x\in C$, so $x=x_\alpha$ for some $\alpha<\kappa$.
	Apply Lemma \ref{lem1} to $x+b$ and $x-b$, so there exist functions $d_0,d_1\colon 2^{\leq n}\rightarrow \bbR$ such that $d_0(\emptyset)=x-b$, $d_1(\emptyset)=x+b$ and conditions \ref{item2}--\ref{item4} with $d_i$ in place of $d$ are fulfilled.
	
	By condition \ref{item3} (with $d$ replaced by $d_0$ and $d_1$) for each $s\in 2^n$ there exist $\alpha_s,\beta_s<\kappa$ and $\delta_s,\varepsilon_s\in\{-1,1\}$ such that
\begin{equation}\label{eq7}
d_0(s)=x_{\alpha_s}+\delta_s b_{\alpha_s}\in Y \text{ and } d_1(s)=x_{\beta_s}+\varepsilon_sb_{\beta_s}\in Y.
\end{equation} By conditions \ref{item1} and \ref{item2} we have
	\begin{equation}\label{eq6}2x_{\alpha}=x+b+x-b=\frac1{2^n}\sum_{s\in 2^n}\left(x_{\alpha_s}+x_{\beta_s}+\delta_s b_{\alpha_s}+\varepsilon_s b_{\beta_s}\right).\end{equation}
	Using condition \eqref{eq2}, we get
	$$\alpha=\alpha_s=\beta_s \text{ for all }s\in 2^n.$$ Otherwise, if $$\gamma=\max\left(\{\alpha\}\cup\{\alpha_s\colon s\in 2^n\}\cup\{\beta_s\colon s\in 2^n\}\right)$$ and we move all $x_\gamma$ to the left hand side and $x_\alpha$ to the right hand side if necessary, we get a contradiction with the choice of $x_\gamma$.

	We have $d_1(s)\in \{x_\alpha-b_\alpha,x_\alpha+b_\alpha\}$ for all $s\in 2^n$, hence by (the last part of) Lemma \ref{lem1}, we get $x+b=x+\varepsilon b_\alpha$ for some $\varepsilon\in\{-1,1\}$. Consequently $b=b_\alpha\notin A$.	
\end{proof}
\section{Bernstein sets and center of distances}

A \textit{Bernstein set} on the real line is a subset $B\subset \bbR$ such that both $B$ and its complement meet every uncountable closed subset of $\bbR$. Any uncountable closed set on the real line has cardinality $\mathfrak c$ and there are exactly continuum many such sets. 
Improving a classical construction of a Bernstein set, we get the following.

\begin{pro}\label{pro1}There exists a Bernstein set $X$ such that $S(X)=[0,\infty)$.
	\end{pro}
\begin{proof}	Let $\{F_\alpha\colon \alpha<\mathfrak{c}\}$ be an enumeration of all uncountable closed subsets of $\bbR$. Choose sequences $(x_\alpha)_{\alpha<\mathfrak{c}}$, $(y_\alpha)_{\alpha<\mathfrak{c}}$ of real numbers such that
	$$x_\alpha\in F_\alpha\setminus \operatorname{lin}_{\mathbb Q}(\{x_\beta\colon \beta< \alpha\}\cup\{y_\beta\colon \beta< \alpha\}),$$
	$$y_\alpha\in F_\alpha\setminus \operatorname{lin}_{\mathbb Q}(\{x_\beta\colon \beta\leq \alpha\}\cup\{y_\beta\colon \beta< \alpha\}).$$
	Put $Y:=\{y_\alpha\colon \alpha <\mathfrak{c}\}$ and $X=\bbR\setminus T^\infty_\emptyset(Y)$. If $a,b\in T_\emptyset^\infty(Y)$, then $\frac12(a+b)\in T_\emptyset^\infty(Y)$, hence $S(X)=[0,\infty)$.
	Also $Y\subset T^\infty_\emptyset(Y)$ and $$T_\emptyset^\infty(Y)\cap \{x_\alpha\colon \alpha<\mathfrak{c}\}\subset\operatorname{lin}_{\mathbb Q}(Y)\cap \{x_\alpha\colon \alpha<\mathfrak{c}\}=\emptyset,$$  so $X\supset \{x_\alpha\colon \alpha<\mathfrak c\}$ is a Bernstein set.
	\end{proof}
Modifying the proof of Theorem \ref{thm1}, we get a final theorem.

\begin{thm}
	\label{thm2}If $0\in A\subset [0,\infty)$, then there exists a Bernstein set $X\subset \bbR$ such that $S(X)=A$.
\end{thm}
\begin{proof}
If $A=[0,\infty)$, then the result follows from Proposition \ref{pro1}.		
Assume $0\in A\subsetneq[0,\infty)$ and consider an enumeration $$\emptyset\neq[0,\infty)\setminus A = \{b_\alpha\colon \alpha<\mathfrak{c}\},$$
	where terms $b_\alpha$ are repeated if necessary.
Let $\{F_\alpha\colon \alpha<\mathfrak{c}\}$ be an enumeration of all uncountable closed subsets of $\bbR$.
	 Inductively choose transfinite sequences $(x_\alpha)_{\alpha<\mathfrak{c}}$, $(y_\alpha)_{\alpha<\mathfrak{c}}$ such that
	\begin{equation}\label{eq3}x_\alpha\in F_\alpha\setminus\operatorname{lin}_{\mathbb Q}\left(\{b_\beta\colon \beta\leq \alpha\}\cup \{x_\beta\colon \beta<\alpha\}\cup\{y_\beta\colon \beta<\alpha\}\right),\end{equation}
	\begin{equation}\label{eq4}y_\alpha\in F_\alpha\setminus\operatorname{lin}_{\mathbb Q}\left(\{b_\beta\colon \beta\leq \alpha\}\cup \{x_\beta\colon \beta\leq\alpha\}\cup\{y_\beta\colon \beta<\alpha\}\right).\end{equation}
	Put  $C:=\{x_\alpha\colon \alpha<\mathfrak c\}$, $Y:=\bigcup_{\alpha<\mathfrak c}\{y_\alpha,x_\alpha+b_\alpha, x_\alpha-b_\alpha\}$ and
	$$X:=\bbR\setminus T_C^\infty\left(Y\right).$$
	
	Since  $C\subset X$  and $\{y_\alpha\colon \alpha<\mathfrak{c}\}\subset T_C^\infty(Y)$, so $X$ is a Bernstein set.
For each $\alpha<\mathfrak c$ the number $x_\alpha\in X$ witnesses that $b_\alpha\notin S(X)$, therefore $S(X)\subset A$.

Fix $b\in [0,\infty)\setminus S(X)$. As in the proof of Theorem \ref{thm1}, check that there exists $\alpha<\mathfrak c$ such that $x_\alpha+b, x_\alpha-b \in T^\infty_C(Y)$ and there exist functions $d_0,d_1\colon 2^{\leq n}\rightarrow \bbR$ such that $d_0(\emptyset)=x_\alpha-b$, $d_1(\emptyset)=x_\alpha+b$ and conditions \ref{item2}--\ref{item4} with $d_i$ in place of $d$ are fulfilled. We have
\begin{equation}\label{eq5}2x_\alpha=\frac1{2^n}\sum_{s\in 2^n}(d_0(s)+d_1(s))\end{equation}
and, by condition \ref{item3}, for each $s\in 2^n$ there are some $\alpha_s,\beta_s<\mathfrak{c}$ such that $$d_0(s)\in \{y_{\alpha_s},x_{\alpha_s}+b_{\alpha_s}, x_{\alpha_s}-b_{\alpha_s}\} \text{ and }d_1(s)\in \{y_{\beta_s},x_{\beta_s}+b_{\beta_s}, x_{\beta_s}-b_{\beta_s}\}.$$ Let
$$\gamma:=\max\left(\{\alpha\}\cup\{\alpha_s\colon s\in 2^n\}\cup\{\beta_s\colon s\in 2^n\}\right).$$
Check that, by \eqref{eq4} and \eqref{eq5}, if $d_0(s)=y_{\alpha_s}$, then $\alpha_s<\gamma$, and similarly, if $d_1(s)=y_{\beta_s}$, then $\beta_s<\gamma$. It follows that $\gamma=\alpha$, because otherwise $\gamma=\alpha_s$ or $\gamma=\beta_s$ for some $s\in 2^n$, and $d_0(s)=x_{\gamma}\pm b_{\gamma}$ or $d_1(s)=x_{\gamma}\pm b_{\gamma}$, which together with \eqref{eq5} would contradict \eqref{eq3}. Finally, again by \eqref{eq3}, $d_0(s)\neq y_{\alpha_s}$, $d_1(s)\neq y_{\beta_s}$ and $\alpha=\alpha_s=\beta_s$  for all $s\in 2^n$ (because otherwise the rational coefficient at $x_{\alpha}$ in \eqref{eq5} would be non-zero). 

Further reasoning proceeds as in the proof of Theorem \ref{thm1}, that is, since $d_1(s)\in \{x_\alpha-b_\alpha,x_\alpha+b_\alpha\}$ for all $s\in 2^n$, we obtain $b=b_\alpha\notin A$.
\end{proof}

\section*{Acknowledgments}
The author thanks the organisers of the 6th Workshop on  Real Analysis in Będlewo (2024), where he heard about the problem, especially Szymon Głąb for a fruitful discussion.

\end{document}